\documentclass{article}
\usepackage{amsmath,amssymb,amsthm}
\newcommand{\cF}{{\mathcal F}}
\newcommand{\Z}{{\mathbb Z}}
\newcommand{\N}{{\mathbb N}}
 
\newcommand{\R}{{\mathbb R}}

\newtheorem{theorem}{Theorem}[section]
\newtheorem{corollary}[theorem]{Corollary}
\newtheorem{lemma}[theorem]{Lemma}
\newtheorem{proposition}[theorem]{Proposition}

\newtheorem{remark}[theorem]{Remark}
\newcommand{\eps}{\varepsilon}
\begin{document}
%% Place the running title of the paper with 40 letters or less in []
 %% and the full title of the paper in { }.
\title{Volume constrained minimizers of the fractional perimeter with a potential energy}

% Place all authors' names in [ ] shown as running head, Leave { } empty
% Please use `and' to connect the last two names if applicable
% Use FirstNameInitial.  MiddleNameInitial. LastName, or only last names of authors if there are too many authors
%\author[A. Cesaroni, M. Novaga]{}

% It is required to enter 2010 MSC.
%\subjclass{ 53A10, 49Q20, 35R11.}
% Please provide minimum  5 keywords.
% Email address of each of all authors is required.
% You may list email addresses of all other authors, separately.

% Put your short thanks below. For long thanks/acknowlegements,
%please go to the last acknowlegments section.

% Add corresponding author at the footnote of the first page if it is necessary. 
% Plase add $^*$ adjacent to the corresponding author's name on the first page. 
% The example shown in this template is if the first author is the corresponding author.

\date{}

\maketitle 

% Enter the first author's name and address:
\centerline{\scshape Annalisa Cesaroni}
\medskip
{\footnotesize
% please put the address of the first author
 \centerline{Department of Statistical Sciences}
   \centerline{University of Padova}
   \centerline{Via Cesare Battisti 141, 35121 Padova, Italy}
  \centerline{email: annalisa.cesaroni@unipd.it}
 } 
 
\medskip

\centerline{\scshape Matteo Novaga}
\medskip
{\footnotesize
 % please put the address of the second  and third author
 \centerline{ Department of Mathematics}
   \centerline{University of Pisa}
   \centerline{Largo Bruno Pontecorvo 5, 56127 Pisa, Italy  }
\centerline{email: novaga@dm.unipi.it}
}

%\thanks{The authors were supported by the Italian GNAMPA and by the University
%of Pisa via grant PRA-2015-0017.}

% The name of the associate editor will be entered by an editorial staff
% "Communicated by the associate editor name" is not needed for special issue.
% \centerline{(Communicated by the associate editor name)}

%The abstract of your paper
\begin{abstract}
\noindent We consider volume-constrained  minimizers of the fractional perimeter with the addition of a 
potential energy in the form of a volume integral. Such minimizers are solutions of the
prescribed fractional curvature problem. 
We prove existence and regularity of minimizers under suitable assumptions on the potential 
energy, which cover the periodic case. 
In the small volume regime we show that minimizers are close to balls, with a quantitative estimate.
\end{abstract}

% \noindent {\bf keywords}: Fractional perimeter, nonlocal isoperimetric problem, 
%periodic medium,  prescribed curvature problem.

%The title of your section 1

%The title of your section 2
\section{Introduction}
Let $s\in (0,1)$ and let $E\subset\R^N$ be a measurable set, the fractional perimeter $P_s(E)$ of $E$ is defined as the squared $H^{s/2}$-seminorm of the characteristic function of $E$, i.e.
\begin{equation}\label{ps}
P_s(E)=\frac{1}{2} \int_{\R^N}\int_{\R^N} \frac{|\chi_E(x)-\chi_E(y)|^2}{|x-y|^{N+s}}dxdy= \int_E\int_{E^c} \frac{1}{|x-y|^{N+s}}dxdy.
\end{equation} 
This notion has been introduced  in \cite{v,crs} and has been widely studied in the last years (see \cite{fffmm,dinoruva} and references therein).  

It is well known that balls are the unique minimizers of the fractional perimeter  among sets with the same volume. Indeed,  the following  fractional  isoperimetric inequality holds for sets of finite volume  (see \cite{crs,fffmm}):  
\begin{equation}\label{isoperim} P_s(E)\geq \frac{P_s(B)}{|B|^{\frac{N-s}{N}}} |E|^{\frac{N-s}{N}}, 
\end{equation} where   $B$ is the  ball of radius $1$, and equality holds if and only if $E$ is a ball. 
The  isoperimetric inequality \eqref{isoperim} can also be localized in bounded sets with Lipschitz boundary (see \cite[Lemma 2.5]{dinoruva}).

In this paper we are interested in existence and properties of minimizers of 
the following isoperimetric problem
\begin{equation}\label{iso}
\min_{|E|=m} \mathcal{F}(E)=\min_{|E|=m} \left( P_s(E)-\int_E g(x)dx\right).\end{equation} In particular we will 
provide regularity properties of minimizers under the assumption that  $g:\R^N\to\R$ is    locally Lipschitz continuous and bounded from above, see Corollary \ref{regcor}, whereas the existence of a solution of the isoperimetric
problem is obtained  for $g$ periodic,  see Theorem \ref{teoex}, or $g$ coercive, that is \begin{equation} \label{coercive}\lim_{|x|\to +\infty} g(x)=-\infty,\end{equation} 
see Proposition \ref{teocoe}.

%We  define  on measurable subsets of $\R^N$ the functional
%\begin{equation}\label{f}
% \mathcal{F}(E)= P_s(E)-\int_E g(x)dx,\end{equation}
Our main result is the following.

\begin{theorem}\label{mainresult} Assume that $g$ is locally Lipschitz and either coercive 
 or $\Z^N$-periodic. 
Then for any $s\in (0,1)$ and $m>0$ there exists a bounded minimizer $E$ of \eqref{iso}. 
Moreover, $\partial E$ is of class $C^{2,\alpha}$ for any $\alpha<s$ outside of a closed 
singular set $S$ of Hausdorff dimension at most $N-3$. 
\end{theorem}

Existence of such minimizers is related  to the problem of finding compact solutions to the geometric equation
\begin{equation}\label{curv} H_s(x)=g(x),\end{equation}
where $H_s$ denotes the $s$-mean curvature at a point $x\in\partial  E $ (see \cite{crs,av})
, that is,  \[H_s(x)= \frac{1}{\omega_{N-2}} \int_{\R^N}\frac{\chi_E(y)-\chi_{E^c}(y)}{|x-y|^{N+s}}dy.\]
Indeed if $E$ is a critical point of the functional \begin{equation}\label{unc} P_s(E)-\int_E g(x)dx, \end{equation} and $\partial E$ is of class 
$C^{1,\alpha}$ for some $\alpha>s$, then it is easy to prove that 
 $E$ solves the  prescribed fractional curvature problem \eqref{curv}. 
Note that in general there is no existence for minimizers of the problem \eqref{unc}, due to the lack of compactness. 

As a corollary of our main result, we get that if  $E$ is a minimizer of \eqref{iso}, then there exists a constant $\mu_m$, depending on $m$, 
such that \[H_s(x)=g(x)+\mu_m\] for $x\in\partial E\setminus S$, where $S$ is the singular 
set in Theorem \ref{mainresult}. 

We will also show in Proposition \ref{smallvolume} that, 
in the small volume regime, the contribution of the volume term $\int_E g(x)dx$ becomes irrelevant, 
and the minimizers converge, after appropriate rescalings, to a ball. 
Note that if $g$ is close to a constant, it is known that solutions to \eqref{curv}
are necessarily compact and close to balls in the Hausdorff distance (see \cite{cfmn}).

\section{Notation and basic estimates} 
Given a set $E\subset \R^N$, we denote as $E^c$ its complement, that is, 
$E^c=\R^N\setminus E$. 
We denote by $B_r$ the ball of center $0$ and radius $r$,  whereas $B(x,r)$ is the ball centered at $x$ and with radius $r$. We also let $\omega_N=|B_1|$. 

Given $E,F$ two sets in $\R^N$,  the symmetric difference of $E$ and $F$ is defined as usual as  $E\Delta F=(E\setminus F)\cup (F\setminus E)$. 

%If $E$ is a set of $\R^N$, we denote as $\partial^\ast E$ the reduced boundary of $E$ and with $\nu_E$ the outer unit normal to $E$.  

We recall   the following computation, that will be useful in the sequel (see \cite[Lemma 2.1]{dinoruva}).  Let $E=E_1\cup E_2$  be a subset of $\R^N$ with $|E_1\cap E_2|=0$, then 
 \begin{equation}\label{psunion}
P_s(E)=P_s(E_1)+P_s(E_2)-2\int_{E_1}\int_{E_2} \frac{1}{|x-y|^{N+s}}dxdy.
\end{equation} 
It is possible to define the nonlocal perimeter of $E$ in a bounded set $\Omega$
as follows: 
\begin{equation}\label{psomega}
P_s(E, \Omega)=  \int_{\R^N\setminus E}\int_{E\cap\Omega} \frac{1}{|x-y|^{N+s}}dxdy +\int_{\Omega\setminus E}\int_{E\setminus\Omega} \frac{1}{|x-y|^{N+s}}dxdy.
\end{equation}
Finally we recall the following formula (see \cite[Lemma 2.4]{dinoruva}). 
Given two disjoint bounded open sets $\Omega_1,\Omega_2$, then there holds 
\begin{equation}\label{psunionlocal}
P_s(E, \Omega_1)+P_s(E,\Omega_2)=P_s(E, \Omega_1\cup \Omega_2)+2\int_{\Omega_1}\int_{\Omega_2} \frac{1}{|x-y|^{N+s}}dxdy.
\end{equation}

\section{Regularity  of minimizers}\label{sectioncompactness}
In this section we shall assume that \begin{equation}\label{g1} \text{ $g$ is locally Lipschitz continuous and 
bounded from above}\end{equation}  and we will prove regularity of minimizers. 

%Before showing existence of minimizers of \eqref{iso}, we prove that every 
%minimizer is necessarily bounded.

We start with 
 a nonlocal version of the so-called Almgren's Lemma (see \cite[Lemma 2.3]{fm}). 
 %Another possibility could be to use a scaling argument (see \cite{gn,dinoruva}), but the periodic volume term in the  energy functional \eqref{f} will be difficult to control.   
\begin{lemma}\label{lemmatranslation} 
Let $s\in (0,1)$ and let $E\subset\R^N$ be a measurable set with $P_s(E)<+\infty$.
Let $x_0\in \R^n$ and $r>0$ be  such that 
\begin{equation}\label{coll}
|B(x_0,r)\cap E|>0 \qquad \text{and}  \qquad |B(x_0,r)\cap E^c|>0\,.
\end{equation} 
Then there exist positive constants $k_0,\,C$, depending on $E$, such that for any 
$k\in (-k_0,k_0)$ there exists a measurable set $F$  with $P_s(F)<+\infty$,
satisfying the following properties
\begin{enumerate}
\item $E\Delta F\subset\subset B(x_0,r)$, 
\item $|F|-|E|=k$, 
%\item $|E\Delta F|\leq C |k|$,
\item $|P_s(E)-P_s(F)|\leq C |k|$.
\end{enumerate}  
\end{lemma}

\begin{proof} 
%First of all we show there exists  $T\in C^\infty_c(B(x_0,1), \R^N)$ such that 
%\[\int_E \text{ div }T(x) dx \neq 0.\]
%It is easy to check that it is possible to choose a point $x_0$ in $E$ such that \[\lim_{r\to 0} |B(x_0,r)\cap E|/|E| =1,\quad |B(x_0,1)\cap E|>0,\quad |B(x_0,1)\cap E^c|>0.\]
%Fix $\delta>0$ such that $(1-\delta)\omega_N > |E\cap B(x_0,1)|$  and choose  $r\in (0,1/2)$ such that $|E\cap B(x_0,r)|\geq (1-\delta)\omega_N r^N$. 
%Let $\phi:[0, +\infty)\to [0, +\infty)$ a smooth function such that \[\phi(t)\equiv \frac{1-r}{r} \  \ t\in [0,r],\quad \phi(t)\equiv 
%0 \ \ t\geq 2r,  \qquad \text{and} \qquad 0<\phi(t)\leq \frac{1}{t} \ \forall t>0.\] 
%We set \[T(y)=x_0+\phi(|x_0-y|)(y-x_0).\] 

Let $T\in C^1_c(B(x_0, r), \R^N)$  be such that 
\[
M:=\int_E \text{ div }T(x)\,dx >0.
\]
Notice that such a vector field $T$ necessarily exists since otherwise we would have
\[
P(E,B(x_0, r)) = \sup \left\{ \int_E \text{ div }T(x)\,dx:\ T\in C^1_c(B(x_0, r), \R^N),\ \|T\|_{\infty}\leq 1\right\} = 0\,,
\]
which, by the relative isoperimetric inequality, would contradict \eqref{coll}.

  %  \[ |T(E)\cap B(x_0,1) | \geq |T(E)\cap B(x_0,1-r)|= |E\cap B_r|/r^n \geq (1-\delta)\omega_n per r<r(\delta)\
%First of all we fix  $T\in C^\infty_c(A, \R^N)$ such that \[\int_{\partial\ast E} T(y)\cdot \nu_E(y) d\mathcal{H}^{n-1}(y)>0.\] Such a function $T$ exists, 
%since $A\cap\partial^\ast E\not=\emptyset$.  Indeed we can choose 
%$x_0\in A\cap\partial^\ast E$ and set $T=\nu_E(x_0) \phi$, where $\phi$ is a nonnegative smooth function such that  supp$\phi\subset B(x_0, 2r)$ 
%and  $\phi\equiv 1$ in $B(x_0, r)$, for $r$ sufficiently small. 

For $t\in (-1,1)$, we define the maps $\Phi_t(x)=x+ t T(x)$. It is easy to see that $\Phi_t$ is a diffeomorphism of $\R^N$ for $t$ sufficiently small, moreover the Jacobian of  $\Phi_t$ is given by  $J\Phi_t(x)=1+t \text{ div }T(x)+o(t)$. 

By construction $E\Delta \Phi_t(E)\subset \subset B(x_0,r)$, moreover
\[|
\Phi_t(E)|=\int_E (1+t \text{ div }T(x)+o(t)) dx
= |E| + Mt +o(t).  
\] 
For $k$ sufficiently small, we then let $F:=\Phi_{t(k)}(E)$ where 
$t(k)=k/M + o(k)$ is such that $|F|=|E|+k$, so that Properties 1 and 2 are verified.

We now compute 
\begin{eqnarray*}P_s(\Phi_t(E))&=&\int_E\int_{E^c} \frac{1+t \text{div}T(x)+t \text{div}T(x)+o(t)}{|\Phi_t(x)-\Phi_t(y)|^{N+s}}dxdy
\\ 
&=&\int_E\int_{E^c} \frac{1+t \text{div}T(x)+t \text{div}T(x)+o(t)}{|x-y+ t(T(x)-T(y))|^{N+s}}dxdy.
 \end{eqnarray*}
Using the regularity of $T$, we get that there exists a constant $C$ (depending on $T$) such that 
\[(1-C|t|)^{N+s}|x-y|^{N+s}\leq |x-y+ t(T(x)-T(y))|^{N+s}\leq (1+C|t|)^{N+s}|x-y|^{N+s}.\]
 Substituting this estimate in the expression for $P_s(\Phi_t(E))$ above, we obtain that
\[ P_s(E) (1-C|t|)\leq P_s(\Phi_t(E)) \leq P_s(E) (1+C|t|) \] where $C$ depends on $T, N, s$. 
This shows that the set $F$ also satisfies Property 3, and the proof is concluded.

%Finally let \[E_t=\{x\in E\ |\ d(x, \partial E\cap A)\leq t \|T\|_\infty\}.\]  
%Then $|E_t|\leq C t$ for some $C$ depending on $\|T\|_\infty$, $E$ and $A$. Observe that  $E\setminus \Phi_t(E) \subset E_t$ and $\Phi_t(E)\setminus E \subset \phi_t(E_t)$. From this we get the desired estimate on $|E\Delta \Phi_t(E)|$. 
\end{proof}
Using this lemma we get boundedness of minimizers. 
 
\begin{proposition}\label{propositioncompact} Let \eqref{g1} hold. 
Then, every  minimizer $E$ of \eqref{iso} is bounded.  
In particular, there exists $\overline R$, depending on $E$, such that 
$E\subseteq B_{\overline R}$, up to a suitable translation.
\end{proposition}

\begin{proof}
Let $E$ be a minimizer of \eqref{iso}. For $r\ge 0$ we
define \[f(r)= |E\setminus B_r|.\] Then $f(r)$ is a nonincreasing function and  by the coarea formula,
we have
\[f'(r)=-P(E\cap B_r).\] 

We claim that there exists $\overline{R}$,  such that $f(r)=0$ for  $r\geq\overline{R}$. 
Let us assume by contradiction that $f(r)>0$ for any $r>0$.  Without loss of generality we can also assume that 
$E\cap B_1\not =\emptyset$ and $ E^c\cap B_1\not =\emptyset$. Moreover, we fix $R_0\ge 1$ such that $f(r)< k_0$ for any $r\geq R_0$, where $k_0$ is as in Lemma \ref{lemmatranslation}.  Then by Lemma \ref{lemmatranslation} for any $r\geq R_0$ there exists a set $F$  such that: 
\begin{enumerate} 
\item $F\Delta E\subset\subset B_1\subseteq B_r$,\item $|F|=|E|+f(r)$, 
\item $|P_s(E)-P_s(F)|\leq C f(r)$.  
\end{enumerate}
Let $G=F\cap B_r$. By the first two properties in Lemma \ref{lemmatranslation}, 
we have that $|G|=|E|$. 
Therefore, by minimality of $E$ and recalling \eqref{psunion}, we get 
\begin{eqnarray}\label{comp}
&& P_s(E)-\int_E g(x)dx<  P_s(G)-\int_G g(x)dx \\ \nonumber 
&& =P_s (F)-P_s(F\setminus B_r)+2 \int_{F\setminus B_r}\int_{F\cap B_r} \frac{1}{|x-y|^{N+s}}dxdy -\int_{F\cap B_r} g(x)dx  .\end{eqnarray} 

By Property 3  in Lemma \ref{lemmatranslation} we get that \begin{equation}\label{uno}P_s(F)\leq P_s(E)+C f(r).\end{equation} 
 
Notice that by  the construction in Lemma \ref{lemmatranslation}, using the locally Lipschitz regularity of $g$, we have
also 
\begin{eqnarray*}
&&\left|\int_{F\cap B_r} g(x)dx
- \int_{E\cap B_r} g(x)dx \right| =\left|\int_{F\cap B_1} g(x)dx
- \int_{E\cap B_1} g(x)dx \right|
\\
&& = 
\left|\int_{E\cap B_1} \big(g(x+t(f(r)))T(x))\,J\Phi_{t(f(r))}(x)-g(x)\big)dx\right| 
\\
\\
&&\leq  (K_g(1) \|T\|_\infty +\|\text{div}T\|_\infty \|g\|_{L^\infty(B_1)}) |E\cap B_1| t(f(r))+ o(f(r)),
\end{eqnarray*}
where $K_g(1)$ is the Lipschitz constant of $g$ in $B_1$.
 
So,  
\begin{equation}\label{due}
- \int_{F\cap B_r} g(x)dx \le  -\int_{E\cap B_r} g(x)dx + C f(r)\leq -\int_{E} g(x)dx +(C+\sup g) f(r).\end{equation}
  
Using the coarea formula and recalling that $E\setminus B_r=F\setminus B_r$, 
we get
\begin{eqnarray}\nonumber
\int_{F\setminus B_r}\int_{F\cap B_r} \frac{1}{|x-y|^{N+s}}dxdy    &\leq &  \int_{E\setminus B_r}\int_{  B_r} \frac{1}{|x-y|^{N+s}}dxdy \\ \label{tre}
\leq \int_{E\setminus B_r}\int_{B^c(y, |y|-r)} \frac{1}{|x-y|^{N+s}}dxdy & =&  \frac{N\omega_N}{s} \int_{E\setminus B_r} \frac{1}{(|y|-r)^{s }} dy  \\ \nonumber
\leq  \frac{N\omega_N}{s}\int_{r}^{+\infty}  \frac{1}{(t-r)^{s }} P(E\cap B_t) dt &=& -\frac{N\omega_N}{s}\int_{r}^{+\infty}  \frac{f'(t)}{(t-r)^{s}}  dt.
\end{eqnarray} 

Substituting \eqref{uno}, \eqref{due}, \eqref{tre} in \eqref{comp},  we eventually obtain
\[P_s (E\setminus B_r)\le
C' f(r)-\frac{N\omega_N}{s}\int_{r}^{+\infty}  \frac{f'(t)}{(t-r)^{s}}dt,\]
for some $C'>0$.
Hence, by the isoperimetric inequality \eqref{isoperim} we get 
\[C(N,s) f(r)^{\frac{N-s}{N}}\le
C' f(r)-\frac{N\omega_N}{s}\int_{r}^{+\infty}  \frac{f'(t)}{(t-r)^{s}}dt.\]

Recalling that $f(r)$ is decreasing to $0$ as $r\to +\infty$, we can choose $R_1>R_0$ such that 
\[
C' f(r)\leq \frac{C(N,s)}{2} f(r)^{\frac{N-s}{N}}
\] 
for all $r\geq R_1$. 
Therefore, for  $r\geq R_1$ we  obtain that  $f$ satisfies the inequality
\begin{equation}\label{ode} 
\frac{s	C(N,s)}{2N\omega_N}f(r)^{\frac{N-s}{N}}\le -\int_{r}^{+\infty}  \frac{f'(t)}{(t-r)^{s}} dt.
\end{equation}
We integrate \eqref{ode} on $(R, +\infty)$, with $R>R_1$, and we exchange the order of integration to get
\begin{equation}\label{ode2}\frac{s	C(N,s)}{2N\omega_N}\int_{R}^{+\infty}f(r)^{\frac{N-s}{N}}dr 
\le -\frac{1}{1-s}\int_{R }^{+\infty}   f'(r) (r-R)^{1-s} dr.
\end{equation}

We now compute 
\begin{eqnarray*}  & & -\frac{1}{1-s}\int_{R }^{+\infty}   f'(r) (r-R)^{1-s} dr\\ & =& - \frac{1}{1-s}\int_{R }^{R+1}   f'(r) (r-R)^{1-s}dr-\frac{1}{1-s}\int_{R+1 }^{+\infty}   f'(r) (r-R)^{1-s}dr\\ 
&\leq &   \frac{f(R)}{1-s} -\frac{f(R+1)}{1-s}  -\frac{1}{1-s}\int_{R +1}^{+\infty}   f'(r) (r-R)^{1-s} dr \\ &=&   \frac{f(R)}{1-s}+\frac{1}{1-s}\int_{R +1}^{+\infty}   f'(r)(1-(r-R)^{1-s})  dr\\ &\leq &  \frac{f(R)}{1-s}+ \int_{R+1 }^{+\infty}   f (r) (r-R)^{ -s} dr  \leq \frac{f(R)}{1-s}+ \int_{R  }^{+\infty}   f (r)   dr.
\end{eqnarray*}
Using again the fact that $f$ is decreasing to $0$, we can choose $R$ sufficiently large such that
\[\int_{R  }^{+\infty}   f (r)   dr\leq \frac{s	C(N,s)}{4N\omega_N}\int_{R  }^{+\infty} f(r)^{\frac{N-s}{N}}  dr.\] 
%Actually  $C(N,s)s$ is a constant independent of $s$, see \cite{fffmm,dinoruva}. 

Substituting this inequality in \eqref{ode2}, we get that $f$ satisfies the integrodifferential inequality
\begin{equation}\label{ode3}\frac{s(1-s)C(N,s)}{4N\omega_N}\int_{R}^{+\infty}f(r)^{\frac{N-s}{N}}dr \le f(R) \end{equation} for all $R\geq R_2$, with $R_2$ sufficiently large. 

Proceeding now exactly as in \cite[Lemma 4.1]{dinoruva},
from \eqref{ode3} we can conclude that there exists $\overline{R}$
such that $f(r)=0$ for every $r\geq \overline{R}$. 
\end{proof}

Once we have boundedness of minimizers, we can obtain regularity. 

We will use  the following result   about regularity of local almost  minimizers of the fractional perimeter, proved in a more general setting in \cite[Thm 1.1, Thm 1.2]{cg}.  
Moreover, in \cite{cg} it is proved that the singular set has Hausdorff dimension at most $N-2$, improved to $N-3$ in \cite[Corollary 2]{sv}. 

\begin{theorem}\label{lemmacg} 
Let $s\in (0,1)$,  $\delta>0$, $\Omega$ an open set. Let $E$ be a  nonlocal almost minimal set. This means  that   for any $x_0\in \partial E$, for any $r<\min(\delta, d(x_0,\partial \Omega))$ and for any measurable set $F$ with $E\Delta F\subset B(x_0, r) $, the following holds \[ P_s(E, \Omega)\leq P_s(F,\Omega) + K r^N.\] Then $E$ has boundary of class $C^1$ outside of a closed singular set $S$ of Hausdorff dimension at most $N-3$. 
\end{theorem} 

We start  showing that any solution  to  the isoperimetric problem \eqref{iso}
is actually also a local minimizer for a suitably defined  unconstrained problem.
\begin{lemma} \label{lemmaunc} Let \eqref{g1} hold. Let $E$ be a minimizer of \eqref{iso} with $|E|=m$. Then
there exists $R >0$ and $\mu_0$, depending on  $E$, such that $E\subseteq B_{R/2}$ 
and $E$ is a solution to 
\[  \min_{F\in B_R} \left(P_s(F)- \int_F g(x)dx +\mu \left||F|-m\right| \right)\] 
for every $\mu\geq \mu_0$. 
\end{lemma}

\begin{proof}
First of all, without loss of generality,  for simplicity  we let $m=1$. 
Let $E$ be a minimizer of $\cF$ among sets $F$ with $|F|=1$. Then, by Proposition 
\ref{propositioncompact} there exist $R$ depending on $E$, $N, s$ and $\|g\|_\infty$ such 
that $E\subseteq B_{R/2}$.

We argue by contradiction and we assume there exists a sequence $\mu_n\to+\infty$ and $F_n\subseteq  B_R$
such that \begin{equation}\label{min1} P_s(F_n)- \int_{F_n} g(x)dx +\mu_n ||F_n|-1| <   P_s(E)-\int_E g(x)dx.\end{equation}  

We  observe that  $||F_n|-1|>0$, since otherwise we would get a contradiction to the previous inequality 
by minimality of $E$ among sets of volume $1$. 

From now on we assume $\mu_n> \|g\|_{L^\infty(B_R)}$ for every $n$. 
We observe that \[\left| \int_{F_n} g(x)dx\right| \leq \|g\|_{L^\infty(B_R)}|F_n|\leq
 \|g\|_{L^\infty(B_R)} ||F_n|-1| +\|g\|_{L^\infty(B_R)}.\] Using this computation and minimality of $F_n$, 
 say \eqref{min1}, we get that there exists $C$ indipendent of $n$ such that 
  \[P_s(F_n)\leq  P_s(E)-\int_E g(x)dx-(\mu_n-\|g\|_{L^\infty(B_R)})||F_n|-1|+ \|g\|_{L^\infty(B_R)}\leq C , \]
and 
  \[(\mu_n-\|g\|_{L^\infty(B_R)})||F_n|-1|\leq P_s(E)-\int_E g(x)dx  +\|g\|_{L^\infty(B_R)}\leq C.\]
 In particular this implies that $|F_n|\to  1$ as $n\to +\infty$. 
 
Let $\lambda_n= |F_n|^{-1/N}$.  Then, by the computation above, $\lambda_n\to 1$ as $n\to +\infty$. 
We define $\tilde F_n= \lambda_n F_n$. So, by definition $|\tilde F_n|=1$ and, by minimality of $E$, we get 
\begin{eqnarray} \nonumber  P_s(E)-\int_E g(x)dx & \leq &   P_s(\tilde F_n) -\int_{\tilde F_n} g(x)dx   =  \lambda_n^{N-s} P_s(F_n) - \lambda_n^N\int_{F_n} g(\lambda_n x) dx \\ \nonumber 
&\leq & \lambda_n^{N-s} P_s(F_n) - \lambda_n^N\int_{F_n} g(x) dx + \lambda_n^N\int_{F_n} |g(\lambda_n x)-g(x)| dx \\ \label{min2} &\leq & \lambda_n^{N-s} P_s(F_n) - \lambda_n^N\int_{F_n} g(x) dx +     R K_g(R) |\lambda_n-1|   
\end{eqnarray} where $K_g(R)$ is the Lipschitz constant of $g$ in $B_R$.  
So, using both \eqref{min1} and \eqref{min2}, we obtain that
\[ \mu_n ||F_n|-1| <(\lambda_n^{N-s}-1) P_s(F_n) - (\lambda_n^N-1)\int_{F_n} g(x) dx +  K_g(R) R |\lambda_n-1|.\]
So, we divide both sides by $||F_n|-1|=|\lambda_n^N-1| \lambda_n^{-N}$ and we obtain, recalling that $P_s(F_n)\leq C$, 
\[\mu_n<   \frac{|\lambda_n^{N-s}-1|}{|\lambda_n^{N }-1|}\lambda_n^N C + |\lambda_n^N-1|\|g\|_{L^\infty(B_R)} +\lambda_n^N\|g\|_{L^\infty(B_R)}+  R K_g(R) \frac{|\lambda_n -1|}{|\lambda_n^{N}-1|}\lambda_n^N.\]
So, in particular, recalling that $\lambda_n\to 1$ as $n\to +\infty$, we get that \[\mu_n\leq C\] for some constant depending on $R,  \|g\|_{L^\infty(B_R)}, K_g(R), N, s$, in contradiction with the assumption that
$\mu_n\to +\infty$. 
\end{proof}
Finally we will use the bootstrap argument in \cite[Theorem 5]{bfv} and the Lipschitz regularity of $g$ 
    to improve the regularity of $\partial E$ from $C^1$ to $C^{2,\alpha}$ for any $\alpha<s$. 
\begin{corollary}\label{regcor} Assume \eqref{g1}.  Let $E$ be a minimizer of \eqref{iso}.  Then $\partial E$  is of class $C^{2,\alpha}$ for every $\alpha<s$, up to a closed singular set $S$ of Hausdorff dimension at most $N-3$. 
\end{corollary} 
\begin{proof} Observe that Lemma  \ref{lemmaunc} implies  that $E$ is a nonlocal almost minimal set in $B_R$.  
Take $\delta<R/2$, $\Omega=B_R$ and $K=(\|g\|_{L^\infty(B_R)}+ \mu_0)\omega_N$. Then
for any $x_0\in \partial E$, for any $r<\delta$ and for any measurable set $F$ with $E\Delta F\subset B(x_0, r) $, the following holds \begin{eqnarray*} P_s(E)&\leq&  P_s(F) +\|g\|_{L^\infty(B_R)} |E\Delta F|+  \mu_0 ||E|-|F|| \\ &\leq &
P_s(F) + (\|g\|_{L^\infty(B_R)}+\mu_0)|E\Delta F|\leq P_s(F) + K r^N.\end{eqnarray*}

Therefore , so we can apply Theorem \ref{lemmacg} and conclude that  $\partial E$  is of class $C^{1}$, 
up to a closed singular set $S$ of Hausdorff dimension at most $N-3$. Eventually we use the bootstrap argument in \cite[Theorem 1.5]{bfv} and the Lipschitz regularity of $g$ 
    to improve the regularity of $\partial E$ from $C^1$ to $C^{2,\alpha}$ for any $\alpha<s$.  \end{proof} 
%%%%%%%%%%%%%%%%%%%%%%%%%%%%%%
\section{Asymptotics of minimizers for small  volumes} \label{sectionas}

In this section we discuss the asymptotic behavior of minimizers of \eqref{iso} in the small volume regime.
We will prove in particular that the volume term becomes irrelevant for small volumes. 

First of all observe that if $E$ is a minimizer of \eqref{iso} with mass constraint $|E|=m$, then $E_\lambda=\lambda E$ is a minimizer of
\begin{equation}\label{rescaled}
\cF_\lambda(E)= P_s (E)- \lambda^{-s} \int_E g\left(\frac{x}{\lambda}\right) dx, 
\end{equation}
among all sets of volume $|E|=\lambda^N m$. 
Indeed $\cF_\lambda(E_\lambda)=\lambda^{N-s}\cF (E)$.

We show  that minimizers of \eqref{iso}, properly rescaled, tend to a ball as the volume goes to zero.

\begin{proposition} \label{smallvolume} Assume that $g\in L^\infty$.
Then for $\eps\in (0,1)$ let $E_\eps$ be a minimizer of \eqref{iso} with volume constraint 
$|E_\eps|=\eps^N \omega_N$, and let $\tilde E_\eps= \eps^{-1} E_\eps$.
Then, as $\eps\to 0$, the sets $\tilde E_\eps$ converge in the $L^1$-topology,  up to translations,
to the unit ball $B$, and in particular there holds 
\begin{equation}\label{eqeq}
\min_{x\in \R^N}\left| \tilde E_\eps \Delta B(x,1)\right| \le C\|g\|_\infty  \eps^s\,,
\end{equation}
where the constant $C$ depends only on $N,\,s$.
\end{proposition}

\begin{proof}
Note that by the observation above $\tilde E_\eps$ is a minimizer of 
the functional $\cF_{\eps^{-1}}$, defined in \eqref{rescaled},
among sets of volume $\omega_N$. 
%The same argument as in Proposition 
%\ref{propositioncompact} shows that there exists $R > 0$ independent of $\eps$ such that 
%all the sets  $\tilde E_\eps$ contained in the $B_R$, up to suitable translations.
%Due to the boundedness of $g$, it is easy to check that the functionals $\cF_{\eps^{-1}}$
%$\Gamma$-converge in the $L^1$-topology to the nonlocal perimeter $P_s$, as $\eps\to 0$.
%As a consequence, we obtain that the minimizers $\tilde E_\eps$  converge in $L^1$ to a minimizer of 
%$P_s$, that is, they converge to a  ball of radius $1$.
By minimality of $\tilde E_\eps$ we then get, for every $x\in \R^N$,
\begin{eqnarray*}
P_s(B)\le P_s(\tilde E_\eps) &\le&  P_s(B(x,1))-\eps^s\int_{B(x,1)} g(\eps y)\,dy +
\eps^s\int_{\tilde E_\eps} g(\eps y)\,dy \\ 
&\le&  P_s(B)-\eps^s\int_{B(x,1)\setminus \tilde E_\eps} g(\eps y)\,dy +
\eps^s\int_{\tilde E_\eps\setminus B(x,1)} g(\eps y)\,dy 
\\ 
&\le&  P_s(B)+\eps^s\| g\|_\infty|\tilde E_\eps\Delta B(x,1)|\,,
\end{eqnarray*}
which gives
\[
P_s(\tilde E_\eps)-P_s(B)\le \eps^s\| g\|_\infty|\tilde E_\eps\Delta B(x,1)|\,.
\]
Recalling the quantitative isoperimetric inequality for the fractional perimeter (see \cite[Thm 1.1]{fffmm}) 
\[
\min_{x\in \R^N}\left| \tilde E_\eps \Delta B(x,1)\right| \le 
C(N,s) \left( P_s(\tilde E_\eps)-P_s(B)\right)^\frac 12\,,
\]
where $C(N,s)$ depends only on $N,\,s$, we then get
\[
\min_{x\in \R^N}\left| \tilde E_\eps \Delta B(x,1)\right| \le 
C(N,s)\|g\|_\infty^{\frac 12} \eps^{\frac s2}\min_{x\in \R^N}\left| \tilde E_\eps \Delta B(x,1)\right|^\frac 12\,,\] from which we obtain \eqref{eqeq}.
\end{proof}

\begin{remark}\rm
The result in Proposition \ref{smallvolume} also holds if  $g$ belongs to $L^\infty_{loc}$ 
and is coercive. %as in \eqref{coercive}. 
Indeed, the proof is the same once we show that the points $x$ in \eqref{eqeq} can be chosen in a fixed compact set,
independent of $\eps$, and this can be easily proved reasoning as in Proposition \ref{teocoe}.
\end{remark}

\section{Existence result}\label{sectionexistence}

We now prove existence of minimizers under suitable assumptions on the function $g$. 

\subsection{Periodic case}The first case  we consider is the case in which $g$  is $\Z^N$ periodic. 
The construction of  a minimizer to \eqref{iso} follows  the same strategy as in the proof of \cite[Theorem 7.2]{dinoruva},  which is based on a concentrated compactness type argument. 

\begin{theorem}\label{teoex} Assume \eqref{g1} and that $g$ is a $\Z^N$ periodic function.  Then,  for every $m>0$ there exists a bounded minimizer $E$   to \eqref{iso}.  
\end{theorem} 
\begin{proof} 
Without loss of generality we shall assume that $m=1/2$, since the argument is the same 
for all values of $m>0$.

We recall  a technical Lemma proved in \cite[Lemma 4.2]{gn}.

\begin{lemma}\label{lemmagn}
Let $C>0$ and let $\{x_i\}_{i\in \mathbb N}$ be a non-increasing sequence of positive numbers such that
\[
\sum_{i=1}^\infty x_i^\frac{N-s}{N} \le C \qquad \text{and}\qquad
\sum_{i=1}^\infty x_i=\frac 1 2\,.
\]
Then there exists $k_0\in\mathbb N$ such that, for all $k\ge k_0$ there holds
\[
\sum_{i=k+1}^\infty x_i \le \frac{1}{(2Ck)^\frac{s}{N}}\,.
\]
\end{lemma}

Let now $E_n$ be a minimizing sequence for \eqref{iso}, that is, 
\[
\lim_{n\to\infty}\cF(E_n)=\inf_{|E|=\frac 12} \cF(E).
\]
In particular, since the function $g$ is bounded, we have
\[
P_s(E_n)\le \cF(E_n) + \int_{E_n} g(x)dx \le C \,,
\]
where $C$ does not depend on $n$. For $n\in\mathbb N$, we also let
$\{Q_{i,n}\}_{i\in\mathbb N}$ be a partition of $\R^N$
into disjoint unit cubes such that the quantities $x_{i,n}=|E_n\cap Q_{i,n}|$ are non-increasing in $i$.
In particular, there holds
\begin{equation}\label{condx}
\sum_{i=1}^\infty x_{i,n} =m= \frac{1}{2}.
\end{equation}
Recalling the fractional isoperimetric inequality \eqref{isoperim}, which can be also localized in 
Lipschitz domains (see \cite[Lemma 2.5]{dinoruva}), we also have
\[
\sum_{i=1}^\infty x_{i,n}^\frac{N-s}{N}\le C  \sum_{i=1}^\infty P_s(E_n, Q_{i,n})
\le 2C P_s(E_n)\le C',
\]
for some constants $C,C'>0$. By Lemma \ref{lemmagn} we then obtain that 
\begin{equation}\label{estx}
\sum_{i=k+1}^\infty x_{i,n} \le c\,k^{-\frac{s}{N}}\,,
\end{equation}
for some $c>0$ and for all $k\in \mathbb N$. By a diagonal argument, up to extracting a subsequence, we can assume that $x_{i,n}\to \alpha_i$ as $n\to\infty$, for some $\alpha_i\in [0,1/2]$.
By \eqref{condx} and \eqref{estx} we then get
\begin{equation}\label{eqstella}
\sum_{i=1}^{\infty} \alpha_i=\frac 12\, .
\end{equation}
Fix now $z_{i,n} \in Q_{i,n}$. Up to extracting a further subsequence, we can suppose that $d(z_{i,n},z_{j,n}) \to
c_{ij} \in [0,+\infty]$, and 
that there exists $G_i\subseteq\R^N$ such that
\begin{equation}\label{limit}
\left(E_{n}-z_{i,n}\right) \to G_i \quad 
\textrm{ in the } L^1_{\rm loc}\textrm{-convergence}
\end{equation}
for every $i\in\mathbb N$. 
We say that $i \sim j$ if $c_{ij} < +\infty$ and we denote by $[i]$ the equivalence class of $i$. Notice that $G_i$
equals $G_j$ up to a translation, if $i\sim j$. Let
$\mathcal A:=\{[i]: i\in\N\}$.
We claim that
\begin{equation}\label{limitbis}
\sum_{[i]\in\mathcal A} P_s(G_i)\le \liminf_{n\to +\infty} P_s(E_n)\,.
\end{equation}
To prove it, we fix~$M\in\N$ and~$R>0$. Let $Q_R=[-R,R]^N$. 
We take different equivalence classes~$i_1,\dots,i_M$
and we notice that if~$i_k\ne i_j$ then the set~$z_{i_k,n}+Q_R$
is moving far apart from the set $z_{i_j,n}+Q_R$, and so we have
$$\lim_{n\to +\infty} 
\int_{z_{i_k,n}+Q_R}\int_{z_{i_j,n}+Q_R}\frac{dx\,dy}{|x-y|^{N+s}}=0.$$
By~\eqref{limit}, the lower semicontinuity of the perimeter and \eqref{psunionlocal},  we obtain
\begin{eqnarray*}
&&\sum_{k=1}^M P_s(G_{i_k}, Q_R)
\le \liminf_{n\to +\infty}\ \sum_{k=1}^M P_s(E_n, (z_{i_k,n}+Q_R))
\\&& \le \liminf_{n\to +\infty} P_s\left(E_n, \bigcup_{k=1}^M (z_{i_k,n}+Q_R)\right)
+ 2\sum_{\frac{1\le k,j\le M}{i_k\neq i_j}} 
\int_{z_{i_k,n}+Q_R}\int_{z_{i_j,n}+Q_R}\frac{dx\,dy}{|x-y|^{N+s}}
\\ && \le \liminf_{n\to +\infty}\ P_s(E_n).
\end{eqnarray*}
By sending first~$R\to +\infty$ and then~$M\to +\infty$, this yields~\eqref{limitbis}.

Now we claim that 
\begin{equation}\label{volume}
\sum_{[i]\in\mathcal A} |G_i|=\frac 12 .
\end{equation}
Indeed, for every $i\in\mathbb N$ and $R>0$ we have
\[
|G_i|\ge |G_i\cap Q_R| = \lim_{n\to +\infty} |(E_n-z_{i,n})\cap Q_R|.
\]
If $j$ is such that $j \sim i$ and $c_{ij} \le \frac{R}{2}$, 
possibly increasing $R$ we have 
$Q_{j,n}- z_{i,n} \subset Q_R$ for all $n\in\mathbb N$, so that
\begin{eqnarray*} 
|(E_n-z_{i,n})\cap Q_R| &=&
\sum_{j=1}^{I_n}|(E_n-z_{i,n})\cap Q_R\cap (Q_{j,n}-z_{i,n})|
\\ &\geq&\sum_{j:\,c_{ij} \leq \frac{R}{2}} 
|(E_n-z_{i,n})\cap Q_R\cap (Q_{j,n}-z_{i,n})|
\\
&=&\sum_{j:\,c_{ij} \leq \frac{R}{2}} 
|(E_n-z_{i,n})\cap (Q_{j,n}-z_{i,n})|\\ &=&
\sum_{j:\,c_{ij} \leq \frac{R}{2}} 
|E_n\cap Q_{j,n}|,\end{eqnarray*}
and so
\[
|G_i|
\ge\lim_{n\to +\infty} \left|\left(E_{n} -z_{i,n}\right) \cap Q_R\right|\geq 
\lim_{n \to +\infty} 
\sum_{j:\,c_{ij} \leq \frac{R}{2}} |E_{n} \cap Q_{j,n}|=\sum_{j:\,c_{ij} \leq \frac{R}{2}} \alpha_j.
\]
Letting $R\to +\infty$ we then have
\[
|G_i| \geq \sum_{j:\,i\sim j} \alpha_j=\sum_{j\in[i]}\alpha_j\,, 
\]
hence, recalling \eqref{eqstella},
\[
\sum_{[i]\in\mathcal A} |G_i| \ge \frac 12 ,
\]
thus proving \eqref{volume} (since the other inequality is trivial).

Let now 
\[
E^{[i]}_n := E_n\cap \bigcup_{j\sim i}Q_{j,n},
\]
and observe that we still have that the sets $(E^{[i]}_n-z_{i,n})$ converge to $G_i$ as
$n\to +\infty$, in the $L^1_{\rm loc}$-convergence. As a consequence, we obtain
\begin{equation}\label{limittris}
\sum_{[i]\in\mathcal A} \int_{G_i} g(x)dx = 
\lim_{n\to +\infty}\sum_{[i]\in\mathcal A} \int_{E^{[i]}_n-z_{i,n}} g(x)dx = 
\lim_{n\to +\infty} \int_{E_n}g(x)dx\,.
\end{equation}
Putting together \eqref{limitbis} and \eqref{limittris} we then get
\begin{equation}\label{limitfour}
\sum_{[i]\in\mathcal A}\cF(G_i)\le \liminf_{n\to +\infty}\cF(E_n)=\inf_{|E|=\frac 12}\cF(E)\,.
\end{equation}
This means in particular that each set $G_i$ is a minimizer of $\cF$ among sets of volume equal to 
$|G_i|$, hence it is bounded thanks to Proposition \ref{propositioncompact}. 

Assume now that at least two of the sets $G_i$'s have positive volume, and let
$F:=\cup_{[i]}(G_i+w_i)$, where the vectors $w_i\in\Z^N$ are chosen in such a way that the sets 
$(G_i+w_i)$ are pairwise disjoint. Then, by $\Z^N$ periodicity of $g$,  and by \eqref{psunion} we get  
\[
\cF(F)<\sum_{[i]\in\mathcal A} ( P_s(G_i+w_i)-\int_{G_i+w_i} g(x)dx) = \sum_{[i]\in\mathcal A}\cF(G_i) \le \inf_{|E|=\frac 12}\cF(E)\,,
\]
thus leading to a contradiction. It follows that there exists $\bar i$ such that $|G_{\bar i}|=1/2$,
so that $G_{\bar i}$ is a (bounded) minimizer of the functional $\cF$.
\end{proof}

\subsection{Coercive case} 
We now assume that $g$ is coercive.  
\begin{proposition}\label{teocoe} Assume that $g$ is a measurable function, bounded from above, and coercive. 
Then for every $m>0$ there exists a minimizer to \eqref{iso}. 
\end{proposition} 
\begin{proof} The argument is the same as for local perimeter functionals (see \cite[Lemma 6]{fm}). 
First of all observe that, up to adding a constant, we can assume that $g\leq 0$.
 
Let $E_n$ be a  minimizing sequence, then
\begin{equation}\label{c1} P_s(E_n)\leq P_s(E_n)-\int_{E_n}g(x)dx \leq C.\end{equation} 
For $R>0$,  we compute
\[ -(\sup_{\R^N\setminus B_R} g)|E_n\setminus B_R| \leq -\int_{E_n\setminus B_R} g(x)dx
\leq P_s(E_n)-\int_{E_n} g(x)dx\leq C.\]
Since by assumption $\sup_{\R^N\setminus B_R} g\to -\infty$ as $R\to +\infty$, this implies that  \begin{equation}\label{c2}\sup_n|E_n\setminus B_R|\to 0\qquad \text{as $R\to +\infty$}.\end{equation}
By \eqref{c1}, \eqref{c2} and the compact embedding of $H^{s/2}$ into $L^1$, there exists a set $E$ with $|E|=m$
such that, up to a subsequence, $E_n\to E$ in $L^1$. 
By the lower  semicontinuity of $P_s$ wth respect to the $L^1$ convergence, it follows that $E$ is a minimizer of \eqref{iso}.
\end{proof}

\end{document}